\documentclass[12pt, twoside]{article}

\newsavebox{\savepar}

\usepackage{amsmath}
\usepackage{amsfonts}
\usepackage{euscript}
\usepackage{oldgerm}
\usepackage{amsthm}
\usepackage{rotating}
\usepackage{mathrsfs}
\usepackage{hyperref}
\usepackage{amssymb}
\usepackage{epsfig}
\usepackage{dsfont}
\usepackage{subfig}

\usepackage[top=1.5in,bottom=1.5in,left=1.2in,right=1.2in]{geometry}

\makeindex

\newtheorem{theorem}{Theorem}[section]
\theoremstyle{definition}

\theoremstyle{remark}

\newtheorem{lemma}[theorem]{\bf Lemma}
\newtheorem{proposition}[theorem]{\bf Proposition}

\theoremstyle{definition}

\theoremstyle{remark}

\numberwithin{equation}{section}

\begin{document}

%\frontmatter
\newcommand{\norm}[1]{\left\lVert #1\right\rVert}
\newcommand{\namelistlabel}[1]{\mbox{#1}\hfil}
\newenvironment{namelist}[1]{%
\begin{list}{}
{
\let\makelabel\namelistlabel
\settowidth{\labelwidth}{#1}
\setlength{\leftmargin}{1.1\labelwidth}
}
}{%
\end{list}}

\newcommand{\K}{\mathcal K}
\newcommand{\inp}[2]{\left\langle {#1} ~ ,\,{#2} \right\rangle}

\newcommand{\vspan}[1]{{{\rm\,span}\{ #1 \}}}
\newcommand{\R} {{\mathbb{R}}}

\newcommand{\C} {{\mathbb{C}}}
\newcommand{\N} {{\mathbb{N}}}
\newcommand{\Q} {{\mathbb{Q}}}
\newcommand{\LL} {{\mathbb{L}}}
\newcommand{\Z} {{\mathbb{Z}}}

\title{Richardson extrapolation for the iterated Galerkin solution of Urysohn integral equations with Green's kernels}
\author{Gobinda RAKSHIT\footnote{School of Mathematical Sciences, UM-DAE Centre for Excellence in Basic Sciences, University of Mumbai Campus, Kalina, Mumbai 400 098, India, email : \textbf{gobinda.rakshit@cbs.ac.in}}, Akshay S. RANE\footnote{Department of Mathematics, Institute of Chemical Technology, Nathalal Parekh Marg, Matunga, Mumbai 400 019, India, email :  \textbf{as.rane@ictmumbai.edu.in}} \: and Kshitij PATIL\footnote{Department of Mathematics, Institute of Chemical Technology, Nathalal Parekh Marg, Matunga, Mumbai 400 019, India, email :  \textbf{kshitijanandpatil@gmail.com}	}
 %The work of this author was partially supported by CEFIPRA,  project no. 4101-1. 
\hspace {1mm}
}
\date{ }
\maketitle

\begin{abstract}
We consider a Urysohn integral operator $\K$ with kernel of the type of Green's function. For $r \geq 1$, a space of piecewise polynomials of degree $\leq r-1 $ with respect to a uniform partition is chosen to be the approximating space and the projection is chosen to be the orthogonal projection. Iterated Galerkin method is applied to the integral equation $x - \K(x) = f$. It is known that the order of convergence of the iterated Galerkin solution is $r+2$ and, at the above partition points it is $2r$.  We obtain an asymptotic expansion of the iterated Galerkin solution at the partition points of the above Urysohn integral equation. Richardson extrapolation is used to improve the order of convergence. A numerical example is considered to illustrate our theoretical results.
\end{abstract}



\thispagestyle{empty}

\newpage

%\setcounter{section}{1}

%\end{document}
\setcounter{equation}{0}
\section{Introduction}

Let $\mathcal{X} = L^\infty [0, 1]$ and 
consider the following  Urysohn integral operator 
\begin{align}\label{eq:1.1}
\mathcal{K} (x)(s)  = \int_0^1 \kappa \left(s, t, x (t)\right)  d t, \;\;\; s \in [0, 1], \; x \in \mathcal{X},
\end{align}
where $\kappa (s, t, u)$ is a real valued continuous function
defined on 
$ \Omega =[0, 1]\times[0, 1] \times \R.$
Then  $\mathcal{K} $ is a compact operator from $L^\infty [0, 1]$ to $C [0, 1].$ 
Consider the Urysohn integral equation
\begin{align}\label{eq:1.2}
	x(s) - \int_0^1 \kappa \left(s, t, x (t)\right)  d t = f(s), \;\;\; s \in [0, 1],
\end{align}
where $f \in \mathcal{X}$ is given and $x$ is the unknown to be determined. We assume that $\varphi$ is an isolated solution of the above equation and consider its numerical approximations.

We are interested in approximate solution which converges to $\varphi$ uniformly. We consider some projection methods associated with a sequence of orthogonal projections converging to the Identity operator point-wise.

In this paper, we consider the case when the kernel $\kappa$ of the integral operator $\K$, is of the type of  Green's function in its domain. We allow the partial derivatives of the kernel $\kappa$ to have jump discontinuities along the diagonal $s = t$. For $ r \geq 1,$ let $ \mathscr{X}_n$ be a space of piecewise polynomials of degree $ \leq r-1 $ with respect to a uniform partition of $ [0, 1]$ with $n$ subintervals each of length ${h = \frac {1} {n} }.$ 
Let $\pi_n$ be the restriction to $L^\infty [0, 1]$ of the orthogonal projection from $L^2 [0, 1]$ onto  $\mathscr{X}_n.$ Galerkin method is a classical projection method for the approximate solution of an integral equation. In this method,
(\ref{eq:1.2}) is approximated by
\begin{equation}\label{eq:Gal}
x_n^G - \pi_n \mathcal {K} (x_n^G) = \pi_n f. 
\end{equation}
The above projection method has been studied extensively in the research literature. 
See Krasnoselsii \cite{Kra}, Krasnoselskii, Vainikko et al \cite{KraV}
and Krasnoselskii and Zabreiko \cite{KraZ} for details. 

\noindent
The iterated Galerkin solution is defined by
$$x_n^S = \K(x_n^G) + f. $$ Note that $$ x_n^G = \pi_n x_n^S $$ and then
the iterated Galerkin solution satisfies the following equation:
\begin{equation}\label{eq:It_Gal}
x_n^S - \K(\pi_{n} x_n^S) = f.
\end{equation}

\noindent
From Atkinson-Potra \cite{Atk-Pot}, we quote the following orders of convergence:

\noindent 
If $r = 1,$ then
\begin{equation}\label{eq1.5}
\| x_n^G - \varphi \|_\infty = O ( h ), \;\;\; 
\| x_n^S  - \varphi \|_\infty = O ( h^{ 2  } ),
\end{equation} 
whereas if $r \geq 2,$ then
\begin{equation}\label{eq:1.5}
\norm{x_n^G - \varphi}_\infty = O\left( h^{r} \right) ~ , \quad \norm{x_n^S - \varphi}_\infty = O\left( h^{r+2}\right). 
\end{equation}

Asymptotic error analysis and extrapolation methods are classical topics in numerical analysis. Richardson extrapolation is the popular one. In Ford et al. \cite{Ford}, Hammerstein integral equation with Green's function type of kernel is considered. Composite trapezoidal quadrature method is used to approximate the integral operator, and then an asymptotic error expansion is obtained for the approximate solution at the node points. Richardson extrapolation is applied to improve the orders of convergence. In Kulkarni-Rane \cite{Rpk-Aks}, the authors have defined the Nystr\"om operator based on the
composite midpoint and the composite modified Simpson rules to approximate the integral operator of a Hammerstein integral equation with Green's function type kernel. Asymptotic expansions for the approximate solution at the node points as well as at the partition points, are obtained and Richardson extrapolation is used to obtain approximate solutions with higher orders of convergence. Hammerstein integral equation is a special case of Urysohn integral equation. The case when the kernel of the Urysohn integral equation is sufficiently smooth, asymptotic error analysis are investigated for various projection methods in Kulkarni-Nidhin \cite{RPK-NTJ}. In the case of a linear integral equation of the second kind with smooth kernel, asymptotic series expansion for the iterated Galerkin solution is proved by McLean \cite{McLean}. The case of asymptotic expansion for approximate solution of integral equations with Green's kernel, at the partition points in the case of Nystr\"{o}m method with midpoint rule, modified Simpsons rule and iterated collocation method is treated in Kulkarni-Rane \cite {Rpk-Aks1}.

In Rakshit-Rane \cite{GR-AR}, we considered a Fredholm integral equation with kernel of the type of Green's function and then asymptotic error analysis is investigated for the iterated Galerkin solution at the partition points. Richardson extrapolation is applied to obtain an approximate solution with higher rate of convergence. 

In this paper we shall analyze asymptotic expansion for the iterated Galerkin solution of Urysohn integral equation with Green's function type kernel. We will use Richardson extrapolation to improve the order of convergence.

The paper is organized as follows. In Section 2, notation is set
and some preliminary results are proved for later use. In Section 3,
asymptotic error analysis for the iterated Galerkin solution at the partition  points are investigated. Numerical illustration is given in Section 4.

\setcounter{equation}{0}
\section{Preliminaries}

In this section we describe the Urysohn integral operator with Green's function type kernel, its Fr\'echet derivatives and related preliminary results. We introduce the following notations.

For an integer $\alpha \geq 0$, let $ C^{\alpha}[0, 1]$ denotes the space of all real valued $\alpha$-times continuously differentiable functions on $[0, 1]$ with the following norm. 
$$ \norm{x}_{\alpha, \infty} = \max_{0 \leq j \leq \alpha} \norm{x^{(j)}}_\infty,$$ where $x^{(j)}$ is the $j^{\text{th}}$ derivative of the function $x$ and 
$$\norm{x^{(j)}}_\infty = \sup_{0 \leq t \leq 1} \left| x^{(j)}(t)\right|.$$
Define
$$ \norm{\kappa}_{\alpha, \infty} = \max_{0 \leq i+j+k \leq \alpha} \norm{D^{(i, j, k)}\kappa(s, t, u)}_\infty,$$
where
$$D^{(i, j, k)}\kappa(s, t, u) = \frac{\partial^{i+j+k}  \kappa}{\partial s^i \partial t^j \partial u^k}(s, t, u).$$

%It follows that $\varphi$ is bounded on $[0, 1]$. Let $$\left[ \inf_{t \in [0, 1]} \left| \varphi(t)\right|, \sup_{t \in [0, 1]} \left| \varphi(t)\right|  \right] \subseteq [a, b].$$

\subsection{Properties of the kernel (Green's function type)}\label{subsection:2.1}

Let $r \geq 1$   be an integer and assume that the kernel $\kappa$ has the following properties.

\begin{enumerate}
	\item For $i = 1, 2, 3, 4$, the functions $\kappa, \displaystyle{\frac { \partial^i \kappa} {\partial u^i} \in C ( \Omega),}$
	where $C ( \Omega)$ denotes the space of all real valued continuous function on $\Omega = [0, 1] \times [0, 1] \times \R$.

	\item Let
	$$ \Omega_1 = \{ (s, t, u): 0 \leq t \leq s \leq 1, \; u \in \R \},\;\;\;
	\Omega_2 = \{ (s, t, u): 0 \leq s \leq t \leq 1, \; u \in \R \}.$$
	There are two functions $\kappa_j \in C^{r} ( \Omega_j ), j = 1, 2, $ such that
	$$
	\kappa (s,t, u) = \left\{ {\begin{array}{ll}
			\kappa_1 (s, t, u), \;\;\; (s, t, u) \in \Omega_1,   \\
			\kappa_2 (s, t, u), \;\;\; (s, t, u) \in \Omega_2.
	\end{array}}\right.
	$$
	
	\item Denote
	$\displaystyle{ \ell (s, t, u) = \frac {\partial \kappa  } { \partial u}( s, t, u)} $ and $\displaystyle{ q (s, t, u) = \frac {\partial^2 \kappa } { \partial u^2}( s, t, u)}, $
	 for all $ (s, t, u) \in \Omega.$ The partial derivatives of $\ell (s, t, u)$ and $q (s, t, u)$ with respect to $s$ and $t$ have jump discontinuities on $s = t$.
		
	\item There are functions $\ell_j, q_j \in C^{r}  ( \Omega_j ), j = 1, 2, $ with
	$$
	\ell (s,t, u) = \left\{ {\begin{array}{ll}
			\ell_1 (s, t, u), \;\;\; (s, t, u) \in \Omega_1,   \\
			\ell_2 (s, t, u), \;\;\; (s, t, u) \in \Omega_2,
	\end{array}}\right.
	$$
	$$
	q (s,t, u) = \left\{ {\begin{array}{ll}
			q_1 (s, t, u), \;\;\; (s, t, u) \in \Omega_1,   \\
			q_2 (s, t, u), \;\;\; (s, t, u) \in \Omega_2.
	\end{array}}\right.
	$$
	
\end {enumerate}
Following Atkinson-Potra \cite{Atk-Pot}, if the kernel $\kappa$ satisfies the above conditions, then we say that $\kappa$ is of class $\mathscr{G}_4(r, 0)$.

	Under the above assumptions, the operator $\mathcal {K}$ is four times Fr\'echet differentiable and its Fr\'echet derivatives at $x \in \mathcal{X}$ are given by 
	$$ \mathcal {K}'(x) v_1 (s) = \int_0^1 \frac {\partial \kappa } {\partial u} \left(s,t, x(t)\right) \:
	v_1(t) \: dt $$
	and 
%	$$ \mathcal {K}''(\varphi) (v_1, v_2)  (s) = \int_0^1 \frac {\partial^2 \kappa } {\partial u^2} \left(s,t,\varphi(t)(t)\right) \:
%	v_1(t) \: v_2(t)\: dt,$$
	\begin{equation}\label{eq:2.1}
	\mathcal {K}^{(i)}(x) (v_1,\ldots, v_i) (s) = \int_0^1 \frac {\partial^i \kappa } {\partial u^i} \left(s,t,x(t)\right) \:
	v_1(t) \cdots v_i(t) \: dt, \qquad i = 2, 3, 4,
	\end{equation}
	where 
	\begin{align*}
	\frac {\partial^i \kappa } {\partial u^i} \left(s,t,x(t)\right) = \frac {\partial^i \kappa } {\partial u^i} \left(s,t,u\right)|_{u = x(t)}, \quad i = 1, 2, 3, 4
	\end{align*} and $v_1, v_2, v_3, v_4 \in \mathcal{X}$. Note that $\mathcal {K}' (x) : \mathcal{X} \rightarrow \mathcal{X}$ is linear and $ \mathcal {K}^{(i)}(x) : \mathcal{X}^i \rightarrow \mathcal{X} $ are multi-linear operators, where $\mathcal{X}^i$ is the cartesian product of $i$ copies of $\mathcal{X}$. See Rall \cite{Rall}.
	 We define 
	\begin{equation}\nonumber
	\norm{\mathcal {K}^{(i)}(x) } \:=\: \sup_{ \stackrel {\norm{v_j}_\infty \leq 1} {j = 1, \ldots, i}} \norm{\mathcal {K}^{(i)}(x) (v_1, \ldots, v_i)}_\infty, \qquad i = 1, 2, 3, 4.
	\end{equation}
	It follows that
	\begin{eqnarray}\nonumber
	\norm{\mathcal {K}^{(i)}(x) } &\leq& \sup_{0 \leq s, t \leq 1} \left| \frac {\partial^i \kappa } {\partial u^i} \left(s, t, x(t)\right) \right|, \qquad i = 1, 2, 3, 4.
	\end{eqnarray}
\noindent
We rewrite the equation \eqref{eq:1.2} as 
\begin{align*}
	x - \mathcal{K} (x) = f, \quad x \in \mathcal{X}.
\end{align*}
Let
\begin{equation}\label{new_op}
	\mathcal{T}(x) = \mathcal{K}(x) + f, \quad x \in \mathcal{X}.
\end{equation}
Assume that $\varphi$ is a fixed point of $\mathcal{T}$. Since $\K$ is compact, $\mathcal {K}' (\varphi)$ is a compact linear operator. See Krasnoselskii \cite{Kra}. Assume that $1$ is not an eigenvalue of $\mathcal {K}' (\varphi).$ Then, $\varphi$ is an isolated solution of \eqref{eq:1.2}. Let $ f \in C^{\alpha} [0, 1]$, then by the Corollary 3.2 of Atkinson-Potra \cite{Atk-Pot}, it follows that $\varphi \in C^{\alpha} [0, 1].$

	\subsection{Approximating Space and Projection operator}
	Let $n \in \mathbb{N}$ and consider the following uniform partition of $[0, 1]:$
	\begin{equation}\label{eq:2.2}
	\Delta: 0  <  \frac{1} {n}  < \cdots <   \frac{n-1} {n}   <  1.
	\end{equation}
	Define
	\begin{equation}\label{partition_points}
	t_j = \frac {j} {n}, \;\;\; j = 0, \ldots, n.
	\end{equation}
	Let 
	\begin{equation}\nonumber
	\Delta_j = [t_{j-1}, t_j]   \;\;\; \mbox {and} \;\;\; h = t_{j} - t_{j-1} = \frac {1} {n}, \;\;\; j = 1, \ldots, n.
	\end{equation}
	Consider a finite dimensional approximating space as
	\begin{equation}\nonumber
	\mathscr{X}_{n} = \left\{ g \in L^{\infty}[0, 1] : g \text{ is a polynomial of degree } \leq r-1 \text{ on }\Delta_j , ~j=1, 2, \dots, n \right\}.
	\end{equation}
	As no continuity conditions  are imposed at the partition points, the dimension of  $\mathscr{X}_n$ is $ n r$ and
	$ \displaystyle {\mathscr{X}_{n} \subset L^{\infty}[0,1]. }$

Let $\pi_n$ be the restriction to $L^\infty [0, 1]$ of the orthogonal projection from $L^2 [0, 1]$ onto  $\mathscr{X}_n$, which converges to the Identity operator pointwise. Then
\begin{equation}\label{eq:2.3}
\sup_{n \geq 1} \norm{\pi_n}_{L^{\infty}[0, 1] \rightarrow L^{\infty}[0, 1]} < \infty.
\end{equation} 
If $x \in C^{\alpha}[0, 1]$, it is well-known that
\begin{eqnarray}\label{eq2.4}
\norm{(I - \pi_{n})x}_\infty & \leq & C_1 \|x^{(\beta)} \|_\infty  h^{\beta}, 
\end{eqnarray}
where $\beta = \min\left\{ \alpha, r \right\}$ and $C_1$ is a constant independent of $h$. See Atkinson \cite{Atk}, Chatelin-Lebbar \cite{CL}.
Denote $$ \pi_{n,j} x = \pi_{n} x |_{\Delta_j}, \quad j = 1, 2, \ldots, n.$$
For $x \in C^{\alpha} (\Delta_j),$ we have 
\begin{eqnarray}\label{eq:2.4}
\|(I - \pi_{n,j})x \|_{\Delta_j, \infty} & \leq & C_2 \|x^{(\beta)} \|_{\Delta_j, \infty}  h^{\beta}, 
\end{eqnarray}
where $\beta = \min\left\{ \alpha, r \right\}$ and $C_2$ is a constant independent of $h$. See Atkinson-Potra \cite[Corollary 4.3]{Atk-Pot}.

\subsection{Asymptotic Expansions and the higher order terms}

Let $\varphi \in C^{2r+2}[0, 1]$. For $\delta > 0,$ let 
$$\mathcal{B}(\varphi, \delta) = \left\{ x \in \mathcal{X} : \norm{x - \varphi}_\infty \leq \delta\right\}$$ denote the closed $\delta$-neighbourhood of $\varphi$. Without loss of generality, we assume that the Galerkin solution $x_n^G$ and the iterated Galerkin solution $x_n^S$ belong to the above neighbourhood.

\noindent
Denote
$$\ell_{*}(s, t) = \frac {\partial \kappa } {\partial u} \left(s,t, \varphi(t)\right), \quad s, t \in [0, 1],$$
$$q_{*}(s, t) = \frac {\partial^2 \kappa } {\partial u^2} \left(s,t, \varphi(t)\right), \quad s, t \in [0, 1].$$
It follows that
\begin{equation}\label{eq:2.5}
\mathcal {K}'(\varphi) v (s) = \int_0^1 \ell_{*}(s, t) \:
v(t) \: dt, \qquad v \in \mathcal{X}, \: s \in [0, 1],
\end{equation}
$$ \K''(\varphi)(v_1, v_2)(s) =  \int_{0}^{1} q_{*}(s,t) ~ v_1(t)\: v_2(t) \: dt, \qquad v_1, v_2 \in \mathcal{X}, \: s \in [0, 1],$$
where the kernels $\ell_{*}(\cdot, \cdot), \; q_{*}(\cdot, \cdot) \in C[0, 1] \times C[0, 1]$ are of the type of Green's function as mentioned in the section \ref{subsection:2.1}. Then 
\begin{equation*}
	\norm{\K'(\varphi)} \leq \sup_{0 \leq t \leq 1} \int_{0}^{1} \left| \ell_{*}(s,t) \right| \: ds, 
\end{equation*}
\begin{equation}\label{eq:2.9}
	\norm{\K''(\varphi)} \leq \sup_{0 \leq t \leq 1} \int_{0}^{1} \left| q_{*}(s,t) \right| \: ds. 
\end{equation}
See Atkinson \cite{Atk}.

By assumption, $ I - \K'(\varphi)$ is invertible. Let
$$ \mathcal{M} = \left( I - \K'(\varphi) \right)^{-1} \K'(\varphi),$$
$$ \mathcal{M}_2 = \left( I - \K'(\varphi) \right)^{-1} \K''(\varphi),$$
$$ \mathcal{M}_3 = \left( I - \K'(\varphi) \right)^{-1} \K^{(3)}(\varphi).$$
Then $\mathcal{M}$, $\mathcal{M}_2 $ and $\mathcal{M}_3$ are respectively compact linear, bi-linear and tri-linear integral operators. See Riesz-Nagy \cite{Riesz-Nagy}. For $v \in \mathcal{X}$, let
\begin{equation*}
\mathcal{M} v (s) = \int_{0}^{1} m(s, t) \: v(t) \: dt, \quad s \in [0, 1],
\end{equation*}
%\begin{equation*}
%	\mathcal{M}_2 (v_1, v_2) (s) = \int_{0}^{1} w(s, t)  \:  v_1(t) \: v_2(t) \: dt, \quad s \in [0, 1]
%\end{equation*}
Note that the kernels of $\mathcal{M}$, $\mathcal{M}_2 $ and $\mathcal{M}_3$ inherit the same smoothness properties as the kernels of $\K'(\varphi)$, $\K''(\varphi)$ and $\K^{(3)}(\varphi)$ respectively. See Atkinson-Potra [Lemma 5.1]\cite{Atk-Pot}. Hence, the kernels of the above three operators are of the type of Green's function as mentioned in section \ref{subsection:2.1}. 

We quote the following result from Rakshit-Rane \cite{GR-AR}.
\begin{equation}\label{asy_exp1}
	\mathcal{M} \varphi(t_i) = \mathcal{M} \pi_{n} \varphi(t_i) + \left(\mathcal{A}_{2r} \varphi \right)(t_i) h^{2r} + O \left( h^{2r+2} \right), \quad i = 0, 1, \ldots, n,
\end{equation}
where
\begin{multline*}
	(\mathcal{A}_{2r} \varphi)(t_i)=\bar{b}_{2r,2r} \int_0^1 m(t_i,t) (t) ~ \varphi^{(2r)}(t) ~ dt \\ ~~~~~+\sum_{p=1}^{2r-1}\bar{b}_{2r,p}  \Bigg\{
	\left[ \left( \frac{\partial}{\partial t}\right) ^{2r-p - 1}\left( m(t_i,t) \varphi^{(p)}(t)\right) \right]_{t=0}^{t=1}  \\
	- \left[ \left( \frac{\partial}{\partial t}\right) ^{2r-p - 1}\left( m(t_i,t) \varphi^{(p)}(t)\right) \right]_{t=t_i-}^{t=t_i+} \Bigg\}
\end{multline*}
with
\begin{equation}\nonumber
	\bar{b}_{2r,p}= \int_{0}^{1} \int_{0}^{1} 
	\Lambda_{r}(\sigma,\tau)\frac{(\sigma-\tau)^{p}}{p!}\frac{B_{2r-p}(\tau)}{(2r-p)!} \: d\sigma \: d\tau,
\end{equation}
$ \displaystyle \Lambda_{r}(\sigma,\tau)= \sum_{q=0}^{r-1}e_{q}(\sigma) e_{q}(\tau)$, $\left\{ e_0, e_1, e_2, \ldots \right\}$ is the sequence of orthonormal polynomials in $L^2[0, 1]$ and $B_{k}$ is the Bernoulli polynomial of degree $k$.

As in Kulkarni-Nidhin \cite[Lemma 2.4]{RPK-NTJ}, it can be shown that for all $s \in [0, 1]$,
\begin{equation}\label{asy_exp2}
	\mathcal{M}_2 (\pi_n \varphi-\varphi)^2 (s) = \mathcal{V}_1 (\varphi) (s) h^{2r} +O(h^{2r+2}) 
\end{equation}
and 
\begin{equation}\label{asy_exp3}
	 \mathcal{M}_3(\pi_n \varphi-\varphi)^3(s)= \mathcal{V}_2 (\varphi)(s) h^{3r} +O(h^{3r+1}),
\end{equation}
where $$ \mathcal{V}_1(\varphi)  = \left( \int_{0}^{1} [\chi_r(t)]^2 dt  \right) \mathcal{M}_2\left(  \varphi^{(r)} \right)^2 ,$$
$$ \mathcal{V}_2(\varphi)  = \left( \int_{0}^{1} [\chi_r(t)]^3 dt  \right) \mathcal{M}_3\left(  \varphi^{(r)} \right)^3 \text{ and }~ \mathcal{V}_2(\varphi)  = 0 \text{ for } ~ r= 1$$ with $$ \chi_r(t) = \int_{0}^{1} \Lambda_{r}(\sigma, t) \frac{(\sigma - t)^r}{r!} d\sigma,$$ are independent of $h$. In the proof of Lemma 2.4 in Kulkarni-Nidhin \cite{RPK-NTJ}, the authors used Euler-McLaurin expansion for smooth kernel. Since, the kernels of $\mathcal{M}_2$ and $ \mathcal{M}_3$ are of the type of Green's function, we use 
%\begin{multline*}
% h \sum_{i=1}^n v[(i-1+\tau)h]  	\\
%=  \int_0^1 v(t) \: dt + B_1(\tau) \left[ v(t) \right]_{t=0}^{t=1} h + \overline{B_1 \left( \tau-\frac{s}{h} \right)}  {\left[ v(t) \right]}_{t=s-}^{t=s+} h + O\left( h^2 \right ), \quad \tau, s \in [0, 1]
%\end{multline*}
the extended Euler-McLaurin summation formula from Kulkarni-Rane \cite{Rpk-Aks1}. It follows that
\begin{equation}\label{eq:2.13}
	\norm{\mathcal{M}_3(\pi_n \varphi-\varphi)^3}_\infty = O\left( h^4 \right ) \quad  \text{ for }~ r=1.
\end{equation}
%The kernels of  $\mathcal{M}_2$ and $ \mathcal{M}_3$  lacks differentiability properties but are continuous on $s = t$. So in the above summation formula, if $v$ is continuous, then
%\begin{equation*}
% h \sum_{i=1}^n \varphi[(i-1+\tau)h] 	=  \int_0^1 \varphi(t) \: dt +  B_1(\tau) \left[\varphi (t)\right]_{t=0}^{t=1} h + O\left( h^2 \right ),
%\end{equation*}
%which is the same formula used in \cite{RPK-NTJ}.

Let
$$ C_3 = \max \left \{ \sup_{\stackrel { 0 \leq t \leq  s \leq 1 }{|u| \leq \|\varphi \|_\infty + \delta }} 
\left | D^{(1,0, 0)} q_{1} (s, t, u)  \right |, \sup_{\stackrel { 0 \leq s \leq  t \leq 1 }{|u| \leq \|\varphi \|_\infty + \delta}} 
\left | D^{(1, 0, 0)} q_{2}  (s, t, u)  \right |   \right \},$$
$$ C_4 = \max \left \{ \sup_ { 0 \leq t \leq  s \leq 1 } 
\left | D^{(1,0)} \ell_{*,1} (s, t)  \right |, \sup_{ 0 \leq s \leq  t \leq 1 } 
\left | D^{(1, 0)} \ell_{*,2} (s, t)  \right |   \right \}.$$
We first prove the following preliminary result which is needed later on.
\begin{lemma}\label{lem:2.1}
	Let $x \in \mathcal{B}(\varphi, \delta)$. Then, for $v_1, v_2 \in \mathcal{X}$, 
	$$\norm{\left( \mathcal{K}''(x)(v_1, v_2) \right)^{'}}_\infty \leq C_3 \norm{v_1}_\infty \norm{v_2}_\infty. $$
\end{lemma}
\begin{proof}
	Let $s \in [0, 1]$. Then
	\begin{align*}
		\mathcal{K}''(x)(v_1, v_2)(s) &= \int_{0}^{1} q(s, t, x(t)) \:v_1(t) \: v_2(t) \: dt \\
		&= \int_{0}^{s} q_1(s, t, x(t)) \:v_1(t) \: v_2(t) \: dt + \int_{s}^{1} q_2(s, t, x(t)) \:v_1(t) \: v_2(t) \: dt.
	\end{align*}
It follows that
\begin{align*}
	\left( \mathcal{K}''(x)(v_1, v_2) \right)^{'}(s) & = \int_{0}^{s} \frac{\partial q_1}{\partial s}\left(s, t, x(t)\right) v_1(t) \: v_2(t) \: dt + q_1(s, s, x(s)) \:v_1(s) \: v_2(s) \\
	& ~~ + \int_{s}^{1} \frac{\partial q_2}{\partial s}\left(s, t, x(t)\right) v_1(t) \: v_2(t) \: dt - q_2(s, s, x(s)) \:v_1(s) \: v_2(s).
\end{align*}
Since $q$ is continuous on $\Omega$,
\begin{align*}
	\left( \mathcal{K}''(x)(v_1, v_2) \right)^{'}(s) =  \int_{0}^{s} \frac{\partial q_1}{\partial s}\left(s, t, x(t)\right) v_1(t)  v_2(t) \: dt  
	 + \int_{s}^{1} \frac{\partial q_2}{\partial s}\left(s, t, x(t)\right) v_1(t)  v_2(t) \: dt.
\end{align*}
Hence,
$$\norm{\left(\mathcal{K}''(x)(v_1, v_2) \right)^{'}}_\infty \leq C_3 \norm{v_1}_\infty \norm{v_2}_\infty.$$
This completes the proof.
\end{proof}

\noindent
Let $x \in \mathcal{B}(\varphi, \delta)$. Then by the above lemma, we obtain 
\begin{equation}\label{eq:2.10}
	\norm{(I - \pi_{n}) \left( \mathcal{K}''(x)(v_1, v_2) \right)}_\infty \leq C_1 C_3 \norm{v_1}_\infty \norm{v_2}_\infty h.
\end{equation}
Similarly, for any $v \in \mathcal{X}$, it can be shown that the function  $\mathcal{K}'(\varphi)v$ is differentiable on $[0, 1]$ and
\begin{equation}\label{eq:2.11}
	\norm{\left(\mathcal{K}'(\varphi)v \right)^{'}}_\infty \leq C_4 \norm{v}_\infty.
\end{equation}
It follows that
\begin{equation}\label{eq:2.12}
	\norm{(I - \pi_{n})\left( \mathcal{K}'(\varphi)v \right)}_\infty \leq C_1 C_4 \norm{v}_\infty  h.
\end{equation}
%Recall that for $v_1, v_2 \in \mathcal{X}$, 
%$$ \K''(\varphi)(v_1, v_2)(s) =  \int_{0}^{1} q_{*}(s,t) \: v_1(t) \: v_2(t) \: dt, $$

The following crucial estimate 
\begin{align}\label{eq12}
\norm{\mathcal{K}'(\varphi)(I - \pi_{n})\varphi}_\infty = \left\{ {\begin{array}{ll}
			 O\left( h^{2}\right), ~~~ ~~ ~ r = 1,   \\
			 O\left( h^{r+2}\right), ~~~\: r \geq 2.
\end{array}}\right.
\end{align}
follows from \textit{Lemma 9} of Chatelin-Lebbar \cite{CL}. From \textit{Theorem 3.1} of Kulkarni \cite{kul}, we have 
\begin{align}\label{eq:2.14}
	\norm{(I - \pi_{n})\mathcal{K}'(\varphi)(I - \pi_{n})\varphi}_\infty = O\left( h^{r+2}\right), \quad r \geq 1.
\end{align}

% Let
%$$C_4 =  \max _{1 \leq j \leq n} \left\{  \sup_{t \in [t_{j-1}, t_j] } \left |D^{(0, r)} q_{*, 1} (s, t)  \right |, \sup_{t \in [t_{j-1}, t_j] } \left |D^{(0, r)} q_{*, 2} (s, t)  \right |\right\}. 
%$$
In order to prove our main result, we need to establish the following lemmas and propositions. Note that
\begin{equation}\label{rel:1}
x_n^G - \varphi = \pi_n(x_n^S - \varphi) - (I - \pi_{n})\varphi.
\end{equation}
We use the above relation between Galerkin and iterated Galerkin solution several times in the following lemmas and propositions.

\begin{lemma}\label{lem:1}
	Let $x_n^G$ be the Galerkin solution defined by the equation \eqref{eq:Gal}. Then for $r \geq 1$, 
	\begin{align}\label{eq5}
	\left( I - \K'(\varphi) \right)^{-1}\mathcal{K}''(\varphi)(x_n^G - \varphi)^2 (s) = (\mathcal{V}_1 (\varphi))(s) \: h^{2r} + O\left( h^{2r+2} \right), \quad s \in [0, 1],
	\end{align}
where $\mathcal{V}_1$ is defined by \eqref{asy_exp2}.
\end{lemma}
\begin{proof}
Using \eqref{rel:1}, we write
\begin{align*}
\K''(\varphi)(x_n^G - \varphi)^2 & = ~ \K''(\varphi) \left[ \pi_n(x_n^S - \varphi) - (I - \pi_{n})\varphi \right]^2 \notag \\
& = ~ \K''(\varphi) \left( \pi_n(x_n^S - \varphi) \right)^2 \notag \\
& ~~~ - 2 ~ \K''(\varphi) \left( \pi_n(x_n^S - \varphi),  (I - \pi_{n})\varphi \right) \notag \\
& ~~~ + ~  \K''(\varphi) \left( (I - \pi_{n})\varphi \right)^2.
\end{align*}	
It follows that
\begin{align}\label{eq6}
	\left( I - \K'(\varphi) \right)^{-1}	\K''(\varphi)(x_n^G - \varphi)^2 
	& = ~ 	\mathcal{M}_2 \left( \pi_n(x_n^S - \varphi) \right)^2 \notag \\
	& ~~~ - 2 ~ 	\left( I - \K'(\varphi) \right)^{-1}\mathcal{K}''(\varphi) \left( \pi_n(x_n^S - \varphi),  (I - \pi_{n})\varphi \right) \notag \\
	& ~~~ + ~  \mathcal{M}_2 \left( (I - \pi_{n})\varphi \right)^2.
\end{align}
Since $\displaystyle{\left( I - \K'(\varphi) \right)^{-1} \K^{(3)}(\varphi)}$ is bounded, from \eqref{eq:1.5}, \eqref{eq:2.3}, \eqref{eq2.4} and \eqref{eq:2.9}, it is easy to see that
\begin{equation}\label{eq7}
	\norm{\mathcal{M}_2 \left( \pi_n(x_n^S - \varphi) \right)^2}_\infty = \left\{ {\begin{array}{ll}
		O\left( h^{4}\right), ~~~~ r = 1,   \\
		O\left( h^{2r+4} \right), ~~~\: r \geq 2.
	\end{array}}\right.
\end{equation}
and
\begin{equation*}
	\norm{\K''(\varphi) \left( \pi_n(x_n^S - \varphi),  (I - \pi_{n})\varphi \right)}_\infty = \left\{ {\begin{array}{ll}
		    O\left( h^{3}\right), ~~~~ r = 1,   \\
		    O\left( h^{2r+2} \right), ~~~\: r \geq 2.
	\end{array}}\right.
\end{equation*}
When $r=1$, that is, when $\mathscr{X}_n$ is the space of piecewise constant functions, the order of the term $\norm{\K''(\varphi) \left( \pi_n(x_n^S - \varphi),  (I - \pi_{n})\varphi \right)}_\infty$ can be improved to $h^4$ in the following way. Note that
\begin{align*}
	\K''(\varphi) \left( \pi_n(x_n^S - \varphi),  (I - \pi_{n})\varphi \right)(s) & = \int_{0}^{1} q_{*}(s,t) \: (\pi_n(x_n^S - \varphi))(t) \: (I - \pi_{n})\varphi(t) \: dt \\
	& =  \sum_{j=1}^{n} \int_{t_{j-1}}^{t_j} q_{*}(s,t) (\pi_{n, j}(x_n^S - \varphi))(t) \: (I - \pi_{n,j})\varphi(t) \: dt.
\end{align*}
Since, the range of $\pi_{n}$ is $\mathscr{X}_n$, $\pi_{n, j}(x_n^S - \varphi)$ is a constant on $[t_{j-1}, t_j]$. It follows that
\begin{multline*}
	\K''(\varphi) \left( \pi_n(x_n^S - \varphi),  (I - \pi_{n})\varphi \right)(s) 
	 = \\  \sum_{j=1}^{n} (\pi_{n, j}(x_n^S - \varphi))\left( \frac{t_{j-1}+t_j}{2} \right) \int_{t_{j-1}}^{t_j} q_{*}(s,t) \: (I - \pi_{n,j})\varphi(t) \: dt.
\end{multline*}
As in the proof of Lemma 9 of Chatelin-Lebbar \cite{CL}, it can be shown that
$$\left| \int_{t_{j-1}}^{t_j} q_{*}(s,t) \: (I - \pi_{n,j})\varphi(t) \: dt \right| =  O\left( h^{3} \right). $$
Since $\left\{ \pi_{n, j} \right\}$ is uniformly bounded and $\norm{x_n^S - \varphi}_\infty = O\left( h^{2} \right)$, from the above estimate it follows that
\begin{equation*}
	\norm{\K''(\varphi) \left( \pi_n(x_n^S - \varphi),  (I - \pi_{n})\varphi \right)}_\infty = O\left( h^{4} \right), \quad r =1.
\end{equation*}
Hence,
\begin{equation}\label{eq:2.23}
	\norm{\K''(\varphi) \left( \pi_n(x_n^S - \varphi),  (I - \pi_{n})\varphi \right)}_\infty = \left\{ {\begin{array}{ll}
			O\left( h^{4}\right), ~~~~~~~ r = 1,   \\
			O\left( h^{2r+2} \right), ~~~\: r \geq 2.
	\end{array}}\right.
\end{equation}

Hence, \eqref{eq5} follows from \eqref{asy_exp2}, \eqref{eq6} and \eqref{eq:2.23}, the proof of the proposition is complete. 
\end{proof}

\begin{lemma}\label{lem:2}
Let $s \in [0, 1]$ and $r \geq 1$. Then
\begin{equation}\label{eq9}
	\left( I - \K'(\varphi) \right)^{-1} \mathcal{K}^{(3)}(\varphi)(x_n^G - \varphi)^3 (s) = \left\{ {\begin{array}{ll}
		O\left( h^{4}\right)	, \qquad \qquad \qquad \qquad ~  r = 1,   \\
		(\mathcal{V}_2(\varphi))(s) \: h^{3r} + O\left( h^{3r+1} \right)	, ~~\: r \geq 2.
	\end{array}}\right.
\end{equation}
where $\mathcal{V}_2$ is defined by \eqref{asy_exp3}.
\end{lemma}
\begin{proof}
We write
\begin{align}\label{eq10}
\left( I - \K'(\varphi) \right)^{-1} \K^{(3)}(\varphi)(x_n^G - \varphi)^3 & = ~ \left( I - \K'(\varphi) \right)^{-1}\K^{(3)}(\varphi) \left[ \pi_n(x_n^S - \varphi) - (I - \pi_{n})\varphi \right]^3 \notag \\
& = ~ \mathcal{M}_3  \left( \pi_n(x_n^S - \varphi) \right)^3 \notag \\
& ~~~ - 3 ~ \mathcal{M}_3 \left( \pi_n(x_n^S - \varphi), \pi_n(x_n^S - \varphi),  (I - \pi_{n})\varphi \right) \notag \\
& ~~~ + 3 ~ \mathcal{M}_3 \left( \pi_n(x_n^S - \varphi), (I - \pi_{n})\varphi, (I - \pi_{n})\varphi \right)\notag \\
& ~~~ - ~ \mathcal{M}_3 \left( (I - \pi_{n})\varphi \right)^3
\end{align}	
Since $\displaystyle{\mathcal{M}_3 = \left( I - \K'(\varphi) \right)^{-1} \K^{(3)}(\varphi)}$ is bounded, \eqref{eq:1.5}, \eqref{eq:2.3}, \eqref{eq2.4} we have
\begin{equation}\label{eq:2.26}
	\norm{\mathcal{M}_3 \left( \pi_n(x_n^S - \varphi) \right)^3}_\infty = \left\{ {\begin{array}{ll}
			O\left( h^{6}\right),	 \qquad  ~  r = 1,   \\
			O\left( h^{3r+6} \right), \quad    r \geq 2.
	\end{array}}\right.
\end{equation}
\begin{equation}\label{eq:2.27}
	\norm{\mathcal{M}_3 \left( \pi_n(x_n^S - \varphi), \pi_n(x_n^S - \varphi),  (I - \pi_{n})\varphi \right)}_\infty = \left\{ {\begin{array}{ll}
			O\left( h^{5}\right),	 \qquad  ~  r = 1,   \\
			O\left( h^{3r+4} \right), \quad    r \geq 2.
	\end{array}}\right.
\end{equation}
\begin{equation}\label{eq:2.28}
	\norm{\mathcal{M}_3 \left( \pi_n(x_n^S - \varphi), (I - \pi_{n})\varphi, (I - \pi_{n})\varphi \right)}_\infty = \left\{ {\begin{array}{ll}
			O\left( h^{4}\right),	 \qquad  ~  r = 1,   \\
			O\left( h^{3r+2} \right), \quad    r \geq 2.
	\end{array}}\right.
\end{equation}
On the other hand, from \eqref{asy_exp3} we have
\begin{equation*}
	\mathcal{M}_3 \left( (I - \pi_{n})\varphi \right)^3 (s) = \left\{ {\begin{array}{ll}
			O\left( h^{4}\right)	, \qquad \qquad \qquad \qquad ~  r = 1,   \\
			(\mathcal{V}_2(\varphi))(s) \: h^{3r} + O\left( h^{3r+1} \right)	, ~~\: r \geq 2.
	\end{array}}\right.
\end{equation*}
Hence, \eqref{eq9} follows from \eqref{eq10}, \eqref{eq:2.26}, \eqref{eq:2.27}, \eqref{eq:2.28} and the above equation. 
\end{proof}

\section{The Main Result}
Recall that the Iterated Galerkin solution is defined by 
\begin{align*}
	x_n^S - \K(\pi_{n} x_n^S) = f
\end{align*}
and the exact solution as
\begin{equation*}
	\varphi - \mathcal{K}(\varphi) = f.
\end{equation*}
In this section, we prove our main result about the asymptotic series expansion for
the iterated Galerkin solution $x_n^S$ at the partition points $t_i, \: i = 0, 1, \dots, n$. That is, we will show the following. 
\begin{align}\label{Sloan_asym}
	\varphi(t_i) - x_n^S (t_i) = \mathcal{A}_{2r}(t_i) h^{2r} + O\left( h^{2r+2} \right),
\end{align}
where $\mathcal{A}_{2r}$ is a function independent of $n$. 

Then, we can apply Richardson extrapolation to obtain an approximation of $\varphi$ with higher order at the partition points. From Ford et al \cite[Section 5]{Ford}, it can be shown that a continuous function can be reconstructed from the extrapolated discrete values at the partition points and it approximates the exact solution $\varphi$ to higher order in the uniform norm. We will not discuss this thing here. Our main aim is to prove \eqref{Sloan_asym}. 

Recall that
$$ \mathcal{M} = \left( I - \K'(\varphi) \right)^{-1}\K'(\varphi)$$
with \begin{align}\label{emm}
	(\mathcal{M}x)(s) = \int_0^1  m(s, t) ~ x (t) ~ d t, \;\;\; s \in [0, 1], \; x \in \mathcal{X}, 
\end{align}
where the kernel $m$ is of the type of Green's function. 

We quote the following expression for the error in the iterated Galerkin solution from Atkinson et al \cite[equation (2.28)]{AGS}:
\begin{align}\label{eq:2.6}
	x_n^S - \varphi & ~=~ \left( I - \K'(\varphi) \right)^{-1} \left\{ \left[  \K(x_n^G) - \K(\varphi) - \K'(\varphi)(x_n^G - \varphi)  \right] \right\}\notag \\
	& ~~~ - \mathcal{M}(I - \pi_{n}) \left[  \K(x_n^G) - \K(\varphi) - \K'(\varphi)(x_n^G - \varphi)  \right] \notag\\
	& ~~~ - \mathcal{M}(I - \pi_{n}) \K'(\varphi)(x_n^G - \varphi) \notag\\
	& ~~~ - \mathcal{M}(I - \pi_{n})\varphi.
\end{align}

By the following propositions, we will prove that the second and the third terms on the right hand side of the above equation are of the order $2r+2$ or higher, and the first and the last term has an asymptotic expansion at the partition points.

 Let 
\begin{eqnarray}\nonumber
	C_5 & = &  \max_{0 \leq i \leq 4} \left( \sup_{ \stackrel {s, t \in [0, 1]} {|u|\leq \|\varphi\|_\infty + \delta}} 
	\left| \frac{\partial^i \kappa}{\partial u^i} (s, t, u) \right | \right).
\end{eqnarray}

Let us investigate the first term on the right hand side of the equation \eqref{eq:2.26} for an asymptotic expansion.
\begin{proposition}\label{prop:1}
Let $x_n^G$ be the Galerkin solution defined by the equation \eqref{eq:Gal} and $s \in [0, 1]$. Then for $r \geq 1,$
%\begin{multline*}
%	\left( I - \K'(\varphi) \right)^{-1}  \left[  \K(x_n^G) - \K(\varphi) - \K'(\varphi)(x_n^G - \varphi)  \right]  (s) \\ =  \left\{ {\begin{array}{ll}
%			\mathcal{V}_1(\varphi)(s) \: h^{2r} +  O\left( h^{4} \right),  \qquad \qquad \qquad \qquad ~ ~  r = 1,   \\
%		\mathcal{V}_1(\varphi)(s) \: h^{2r} + \mathcal{V}_2(\varphi)(s) \: h^{3r} +  O\left( h^{2r+2} \right)	, ~~\: r \geq 2.
%	\end{array}}\right.
%\end{multline*} 
	\begin{equation*}
		\left( I - \K'(\varphi) \right)^{-1}  \left[  \K(x_n^G) - \K(\varphi) - \K'(\varphi)(x_n^G - \varphi)  \right]  (s) = \mathcal{V}_1(\varphi)(s) \: h^{2r} + O\left( h^{2r+2} \right),
	\end{equation*}
where $\mathcal{V}_1$ is defined by \eqref{asy_exp2}.
\end{proposition}
\begin{proof}
	 Using the generalized Taylor's series expansion (see Linz \cite{Linz}) in the neighbourhood $\mathcal{B}(\varphi, \delta)$, we obtain
	\begin{align*}
		\K(x_n^G) - \K(\varphi) - & \K'(\varphi)(x_n^G - \varphi)  \\ 
		& = \frac{1}{2} \K''(\varphi)(x_n^G - \varphi)^2 + \frac{1}{6} \K^{(3)}(\varphi)(x_n^G - \varphi)^3  + \mathcal{R}_4 \left( x_n^G, \varphi \right),
	\end{align*}
where
$$\mathcal{R}_4 \left( x_n^G, \varphi \right) = \frac {1} {6} \int_0^1 \mathcal{K}^{(4)} \left(\varphi + \theta (x_n^G - \varphi) \right) (x_n^G - \varphi)^4
(1 - \theta)^3 \; d \theta.$$
By \eqref{eq:2.1}, we have
	\begin{align*}
		\mathcal{K}^{(4)} \left(\varphi + \theta (x_n^G - \varphi) \right) (x_n^G - \varphi)^4 (s) = \int_{0}^{1} \frac{\partial^4 \kappa}{\partial u^4} \left(s, t, \varphi(t) + \theta (x_n^G - \varphi)(t) \right)(x_n^G - \varphi)^4 (t) \: dt.
	\end{align*}
	Since $\norm{x_n^G - \varphi}_\infty \rightarrow 0$ as $n \rightarrow \infty$ and $\theta \in (0, 1)$, $$\varphi + \theta (x_n^G - \varphi) \in \mathcal{B}(\varphi, \delta).$$ It follows that
	\begin{eqnarray}\nonumber
		\norm{\mathcal{K}^{(4)} \left(\varphi + \theta (x_n^G - \varphi) \right) (x_n^G - \varphi)^4}_\infty \leq C_5 \norm{x_n^G - \varphi}_\infty^4
	\end{eqnarray}
	Using \eqref{eq:1.5} and the above estimate, we obtain
	\begin{equation}\label{eq:2.7}
	\mathcal{R}_4 \left( x_n^G, \varphi \right)	= O\left( h^{4r} \right).
	\end{equation}
Note that
\begin{align*}
	\left( I - \K'(\varphi) \right)^{-1}  & \left[  \K(x_n^G) - \K(\varphi) - \K'(\varphi)(x_n^G - \varphi)  \right] \\  
	= & \frac{1}{2} \: \left( I - \K'(\varphi) \right)^{-1} \K''(\varphi) (x_n^G - \varphi)^2 + \frac{1}{6} \: \left( I - \K'(\varphi) \right)^{-1} \K^{(3)}(\varphi)(x_n^G - \varphi)^3  \\
	& ~ +   \left( I - \K'(\varphi) \right)^{-1} \mathcal{R}_4 \left( x_n^G, \varphi \right).
\end{align*}
Hence, by Lemma \ref{lem:1}, Lemma \ref{lem:2}, equation \eqref{eq:2.7} and and the above estimate, we obtain the followings. \\
For $r = 1$,
\begin{equation*}
	\left( I - \K'(\varphi) \right)^{-1}  \left[  \K(x_n^G) - \K(\varphi) - \K'(\varphi)(x_n^G - \varphi)  \right]  (s)  = \mathcal{V}_1(\varphi)(s) \: h^{2r} +  O\left( h^{4} \right),
\end{equation*}
for $r \geq 2$,
\begin{multline*}
	\left( I - \K'(\varphi) \right)^{-1}  \left[  \K(x_n^G) - \K(\varphi) - \K'(\varphi)(x_n^G - \varphi)  \right]  (s)  \\ = \mathcal{V}_1(\varphi)(s) \: h^{2r} +  \mathcal{V}_2(\varphi)(s) \: h^{3r} + O\left( h^{2r+2} \right).
\end{multline*}
Since $3r \geq 2r+2$ for $r \geq 2$, 
\begin{equation*}
	\left( I - \K'(\varphi) \right)^{-1}  \left[  \K(x_n^G) - \K(\varphi) - \K'(\varphi)(x_n^G - \varphi)  \right]  (s)  = \mathcal{V}_1(\varphi)(s) \: h^{2r} +  O\left( h^{2r+2} \right), \quad r \geq 2.
\end{equation*}
This completes the proof.
\end{proof}

 Now we investigate the second term on the R.H.S. of the equation \eqref{eq:2.6}.
\begin{proposition}\label{prop:2} Let $\displaystyle{ \left\{t_i : i = 0, 1, \dots, n \right\}}$ be the set of all partition points defined by \eqref{partition_points}. Then for $r \geq 1$,
\begin{equation*}
	\mathcal{M}(I - \pi_{n}) \left[  \K(x_n^G) - \K(\varphi) - \K'(\varphi)(x_n^G - \varphi)  \right](t_i) = O\left( h^{2r+2} \right).
\end{equation*}	
\end{proposition}
\begin{proof}
By the generalized Taylor's theorem we obtain
\begin{equation*}
	\K(x_n^G) - \K(\varphi) - \K'(\varphi)(x_n^G - \varphi) = 	\int_0^1 \mathcal{K}''\left(\varphi + \theta (x_n^G - \varphi) \right) (x_n^G - \varphi)^2
	(1 - \theta) \; d \theta.
\end{equation*}
Let
\begin{equation}\label{eq:3.4}
	\mathcal{R}_2(x_n^G, \varphi) = \int_0^1 \mathcal{K}''\left(\varphi + \theta (x_n^G - \varphi) \right) (x_n^G - \varphi)^2
	(1 - \theta) \; d \theta.
\end{equation}
Note that
\begin{align*}
	\mathcal{K}'' \left(\varphi + \theta (x_n^G - \varphi) \right) (x_n^G - \varphi)^2 (s) = \int_{0}^{1} \frac{\partial^2 \kappa}{\partial u^2} \left(s, t, \varphi(t) + \theta (x_n^G - \varphi)(t) \right)(x_n^G - \varphi)^2 (t) \: dt.
\end{align*}
%Since $\norm{x_n^G - \varphi}_\infty \rightarrow 0$ as $n \rightarrow \infty$ and $\theta \in (0, 1)$, $$\varphi + \theta (x_n^G - \varphi) \in \mathcal{B}(\varphi, \delta).$$
 It follows that
\begin{eqnarray}\nonumber
	\norm{\mathcal{K}'' \left(\varphi + \theta (x_n^G - \varphi) \right) (x_n^G - \varphi)^2}_\infty \leq C_3 \norm{x_n^G - \varphi}_\infty^2.
\end{eqnarray}
Since $\displaystyle{\norm{x_n^G - \varphi}_\infty =  O\left( h^{r} \right)}$,
\begin{eqnarray}\label{eq:3.5}
	\norm{\mathcal{K}'' \left(\varphi + \theta (x_n^G - \varphi) \right) (x_n^G - \varphi)^2}_\infty = O\left( h^{2r} \right). 
\end{eqnarray}
Let $s \in [0, 1]$ be fixed and $m_{s}(t) = m(s, t) , ~  t \in [0, 1]$, then
\begin{multline*}
	\mathcal{M}(I - \pi_{n}) \left[  \K(x_n^G) - \K(\varphi) - \K'(\varphi)(x_n^G - \varphi)  \right](s) \\
	= \left< m_s, \: (I - \pi_{n}) \left[  \K(x_n^G) - \K(\varphi) - \K'(\varphi)(x_n^G - \varphi)  \right] \right>, 
\end{multline*}
where $\left< \cdot, \cdot \right>$ is the usual inner product in $L^2[0, 1]$, i.e., $$\left< x, y \right> = \int_{0}^{1} x(t)\: y(t) \: dt, \quad x, y \in L^2[0, 1].$$
Since $I - \pi_{n}$ is self-adjoint, 
\begin{multline}\label{eq:3.7}
	\mathcal{M}(I - \pi_{n}) \left[  \K(x_n^G) - \K(\varphi) - \K'(\varphi)(x_n^G - \varphi)  \right](t_i) \\
	= \left< (I - \pi_{n}) m_{t_i}, \:   (I - \pi_{n}) \left[  \K(x_n^G) - \K(\varphi) - \K'(\varphi)(x_n^G - \varphi)  \right]   \right>.
\end{multline}
Note that $m_{t_i}$ is continuous on $[t_{j-1}, t_j]$ and $r$ times continuously differentiable on $(t_{j-1}, t_j)$ for all $j=1, 2, \dots, n$. Therefore by \eqref{eq:2.4}	
\begin{align}\label{eq3.8}
	\norm{(I - \pi_{n, j})m_{t_i}}_{\Delta_j, \infty} =  O\left( h^{r} \right).
\end{align}
Using \eqref{eq:2.3}, \eqref{eq:3.5}, \eqref{eq:3.7} and the above estimate, we obtain
\begin{equation}\label{eq:12}
	\mathcal{M}(I - \pi_{n}) \left[  \K(x_n^G) - \K(\varphi) - \K'(\varphi)(x_n^G - \varphi)  \right](t_i) = O\left( h^{3r} \right), \quad r \geq 2.
\end{equation}

Consider the case $\mathbf{r =1}$.

From \eqref{eq:3.4}, it is easy to see that if $\mathcal{K}'' \left(\varphi + \theta (x_n^G - \varphi) \right) (x_n^G - \varphi)^2$ is differentiable, then $\mathcal{R}_2(x_n^G, \varphi)$ is differentiable. Since $\varphi + \theta (x_n^G - \varphi) \in \mathcal{B}(\varphi, \delta)$, from Lemma \ref{lem:2.1} we have $\mathcal{R}_2(x_n^G, \varphi)$ is differentiable. Thus, from \eqref{eq:3.4}
\begin{equation*}
	\left( \mathcal{R}_2(x_n^G, \varphi) \right)^{'} = \int_0^1 \left( \mathcal{K}''\left(\varphi + \theta (x_n^G - \varphi) \right) (x_n^G - \varphi)^2 \right)^{'} 
	(1 - \theta) \; d \theta.
\end{equation*}
This implies,
\begin{equation*}
	\norm{\left( \mathcal{R}_2(x_n^G, \varphi) \right)^{'}}_\infty \leq \frac{1}{2} \norm{\left( \mathcal{K}''\left(\varphi + \theta (x_n^G - \varphi) \right) (x_n^G - \varphi)^2 \right)^{'}}_\infty, \qquad 0 < \theta < 1.
\end{equation*}
Using Lemma \ref{lem:2.1}, 
\begin{equation*}
	\norm{\left( \mathcal{R}_2(x_n^G, \varphi) \right)^{'}}_\infty \leq  C_3 \norm{x_n^G - \varphi}^2.
\end{equation*}
From \eqref{eq2.4}, it follows that
\begin{align*}
	\norm{(I - \pi_{n}) \left[  \K(x_n^G) - \K(\varphi) - \K'(\varphi)(x_n^G - \varphi)  \right]}_\infty  & ~ = ~ \norm{(I - \pi_{n}) \mathcal{R}_2(x_n^G, \varphi)}_\infty \\ 
	& ~ \leq ~ C_1 C_3 \norm{x_n^G - \varphi}^2 h.
\end{align*}
By \eqref{eq1.5}, \eqref{eq:3.7}, \eqref{eq3.8} and the above estimate, we obtain
\begin{equation}\label{eq:3.8}
 \mathcal{M}(I - \pi_{n}) \left[  \K(x_n^G) - \K(\varphi) - \K'(\varphi)(x_n^G - \varphi)  \right](t_i) = 	O\left( h^{4} \right), \quad r = 1.
\end{equation}
Hence, the required result follows from \eqref{eq:12} and \eqref{eq:3.8}.	
\end{proof}

 Next we investigate the third term on the R.H.S. of the equation \eqref{eq:2.6}. 

\begin{proposition}\label{prop:3}
Let $\displaystyle{ \left\{t_i : i = 0, 1, \dots, n \right\}}$ be the set of all partition points defined by \eqref{partition_points}. Then $$   \mathcal{M}(I - \pi_{n}) \K'(\varphi)(x_n^G - \varphi)(t_i)  = O\left( h^{2r+2} \right), \quad \text{ for } r \geq 1 . $$	
\end{proposition}
\begin{proof}
	From \eqref{emm} we have
	\begin{align}\label{eq13}
	\mathcal{M}(I - \pi_{n}) \K'(\varphi)(x_n^G - \varphi)(t_i)  & = \int_0^1  m(t_i, t) ~ (I - \pi_{n}) \K'(\varphi)(x_n^G - \varphi)(t) ~ d t \notag\\
	& = \left< m_{t_i} ~,~ (I - \pi_{n}) \K'(\varphi)(x_n^G - \varphi) \right> \notag\\
	& = \left< (I - \pi_{n})m_{t_i} ~, ~ (I - \pi_{n})\K'(\varphi)(x_n^G - \varphi) \right>.
	\end{align} 
It is easy to see that, $m_{t_i}$ is continuous on $[t_{j-1}, t_j]$ and $r$ times continuously differentiable on $(t_{j-1}, t_j)$ for all $j=1, 2, \dots, n$. Therefore by \eqref{eq:2.4}	
\begin{align}\label{eq14}
\norm{(I - \pi_{n, j})m_{t_i}}_{\Delta_j, \infty} =  O\left( h^{r} \right).
\end{align}
Note that 
\begin{equation}\label{eq3.9}
	 \K'(\varphi)(x_n^G - \varphi) =  \K'(\varphi)\left( \pi_{n}(x_n^S - \varphi) \right) -   \K'(\varphi)(I - \pi_{n})\varphi.
\end{equation}
Thus, by \eqref{eq:1.5}, \eqref{eq:2.3} and \eqref{eq12} we obtain, $$ \norm{ \K'(\varphi)(x_n^G - \varphi)}_\infty = O\left( h^{r+2} \right), \quad r \geq 2. $$
Hence, from \eqref{eq13}, \eqref{eq14} and the above estimate, we obtain
\begin{equation}\label{eq:3.10}
	 \mathcal{M}(I - \pi_{n}) \K'(\varphi)(x_n^G - \varphi)(t_i)  = O\left( h^{2r+2} \right) , \quad r \geq 2.
\end{equation}

When $\mathbf{r =1}$, that is, when $\mathscr{X}_n$ is the space of piecewise constant functions, it is easy to see from \eqref{eq1.5}, \eqref{eq12} and \eqref{eq3.9} that
$$ \norm{ \K'(\varphi)(x_n^G - \varphi)}_\infty = O\left( h^{2} \right),$$ which is not equal to $O\left( h^{2r+2} \right)$ with $r=1$. We consider this case separately.

Note that,  
\begin{equation}\label{eq:3.11}
	(I - \pi_{n}) \K'(\varphi)(x_n^G - \varphi) = (I - \pi_{n}) \K'(\varphi)\left( \pi_{n}(x_n^S - \varphi) \right) -  (I - \pi_{n}) \K'(\varphi)(I - \pi_{n})\varphi.
\end{equation}
From \eqref{eq:2.12}, we have
\begin{equation*}
	\norm{(I - \pi_{n}) \K'(\varphi)\left( \pi_{n}(x_n^S - \varphi) \right)}_\infty \leq C_1 C_4 \norm{\pi_n} \norm{x_n^S - \varphi}_\infty h.
\end{equation*}
From \eqref{eq1.5} and \eqref{eq:2.3}, it follows that
\begin{equation*}
	\norm{(I - \pi_{n}) \K'(\varphi)\left( \pi_{n}(x_n^S - \varphi) \right)}_\infty = O\left( h^{3} \right).
\end{equation*}
By \eqref{eq:2.14}, \eqref{eq13}, \eqref{eq14}, \eqref{eq:3.11} and the above estimate, we obtain
\begin{equation}\label{eq:3.12}
	\mathcal{M}(I - \pi_{n}) \K'(\varphi)(x_n^G - \varphi)(t_i)  = O\left( h^{4} \right) , \quad r = 1.
\end{equation}
Hence, the required result follows from \eqref{eq:3.10} and \eqref{eq:3.12}.
\end{proof}

\noindent
Now, we prove our main theorem.

\begin{theorem}\label{thm:1} Let $f \in C^{2r}[0, 1]$, and the kernel of the Urysohn integral operator \eqref{eq:1.1} be of class $\mathscr{G}_4(r, 0)$. Let $\varphi$ be a fixed point of the operator $\mathcal{T}$ defined by \eqref{new_op}, with $1$ not an eigenvalue of $\mathcal{K}'(\varphi)$. For $r \geq 1$, let $\mathscr{X}_n$ be the space of piecewise polynomials of degree $ \leq r-1 $ with respect to the partition \eqref{eq:2.2} and $\pi_{n}$ be the orthogonal projection defined by \eqref{eq:2.3}--\eqref{eq2.4}. Let $x_n^S$ be the iterated Galerkin solution defined by \eqref{eq:It_Gal}. Then, for	$i = 0, 1, \ldots, n$,
	\begin{equation*}
		x_n^S(t_i) - \varphi(t_i) = - \zeta_{2r}(t_i) \: h^{2r} + O\left( h^{2r+2} \right), 
	\end{equation*}
where $\zeta_{2r}$ is a function bounded by a constant independent of $h$.
\end{theorem}

\begin{proof}
	From the equation \eqref{eq:2.6}
	\begin{align*}
		x_n^S(t_i) - \varphi(t_i) & ~=~ \left( I - \K'(\varphi) \right)^{-1}  \left[  \K(x_n^G) - \K(\varphi) - \K'(\varphi)(x_n^G - \varphi)  \right](t_i) \notag \\
		& ~~~ - \mathcal{M}(I - \pi_{n}) \left[  \K(x_n^G) - \K(\varphi) - \K'(\varphi)(x_n^G - \varphi)  \right](t_i) \notag\\
		& ~~~ - \mathcal{M}(I - \pi_{n}) \K'(\varphi)(x_n^G - \varphi)(t_i) \notag\\
		& ~~~ - \mathcal{M}(I - \pi_{n})\varphi(t_i).
	\end{align*}

Let \begin{align*}
	\zeta_{2r} = \mathcal{A}_{2r} + \mathcal{V}_1(\varphi).
\end{align*}
	Hence, the proof of this theorem follows from the equation \eqref{asy_exp1}, Proposition \ref{prop:1}, Proposition \ref{prop:2} and Proposition \ref{prop:3}.
	
\end{proof}

We can now apply one step of Richardson extrapolation and obtain an approximations of $\varphi$ of order $h^{2r+2}$ at the partition points.

Define
\begin{equation*}
	 x_n^{EX} = \frac { 2^{ 4 r } x_{2 n}^S -  x_n^S} { 2^{ 4 r } - 1}.
\end{equation*}
Then under the assumptions of Theorem \ref{thm:1}, we have the following result
\begin{equation}  \label{eq:3.9}
	x_n^{EX} {(t_i)} - \varphi{(t_i)}   = O \left( h^{ 2r +2} \right), \quad i = 0, 1, 2, \ldots, n.
\end{equation}

%%%%%%%%%%%%%%%%%%%%%%%%%%%%%%%% Numerical Example %%%%%%%%%%%%%%%%%%%%%%%%%%%%%%%

\section{Numerical Illustration}

For the sake of numerical illustration, we consider the following example of a non-linear Hammerstein integral equation from Kulkarni-Rane \cite{Rpk-Aks}. 
\noindent
Consider
\begin{equation}\label{eq:4.1}
	\varphi (s) - \int_0^1 \kappa (s, t) \left[ \psi  \left(t, \varphi (t) \right) \right] \: dt  = f(s), \;\;\; 0 \leq s \leq 1,
\end{equation}
where 
$$
\kappa (s,t) =\frac{1}{\gamma \sinh \gamma} \left\{ {\begin{array}{ll}
		\sinh \gamma s \: \sinh \gamma(1-t), & ~ 0 \leq t \leq s \leq 1, \\
		\gamma (1-s) \sinh \gamma t, & ~ 0 \leq s \leq t \leq 1,
\end{array}}\right. 
$$
with $\gamma = \sqrt{12},$
and $$  \psi(t, \varphi(t))= \gamma^2 \varphi(t) - 2 \left( \varphi(t) \right)^3, \quad t \in [0,1].$$
We have $f(s) =\frac{1}{\gamma \sinh \gamma} \left \lbrace 2 \sinh \gamma(1-s) + \frac{2}{3} \sinh \gamma s \right \rbrace. $ The exact solution of \eqref{eq:4.1} is given by 
\begin{align*}
	\varphi(s) =\frac{2}{2s+1}, \quad s \in [0, 1].
\end{align*}

 Let $\mathscr{X}_n$ be the space of piecewise constant polynomials with respect to the uniform partition \eqref{eq:2.2} of the interval $[0,1]$ considered before. Let $\pi_n: L^\infty[0,1] \rightarrow \mathscr{X}_n$ be the orthogonal projection defined by \eqref{eq:2.3}--\eqref{eq2.4}.\\ In this case, it is given by $$ (\pi_n \varphi)(s) =\displaystyle \frac{1}{h} \int_{(i-1)h}^{ih} \varphi(t) \: dt,\; s \in [(i-1)h,ih],$$ where $h = \frac{1}{n}$. In the definition of projection operator defined above, if we replace the integral by the right hand rule, then 
$ (\pi_n \varphi)(ih-)= \varphi(ih-).$ Let $x_n^S$ be the Sloan solution defined by \eqref{eq:It_Gal}. Then $x_n^G(ih)=\displaystyle  \frac{\pi_n x_n^S(ih-)+\pi_n x_n^S(ih+)}{2}$ at the partition points is obtained by solving the approximate system of non-linear equations which gives the values of $\pi_n x_n^S(ih-)$ and $\pi_n x_n^S(ih+).$ The system is as follows:
\begin{equation}
\alpha_j = h \sum_{l=1}^n \kappa(s_j,s_l) \left[ \gamma^2 \alpha_l - 2\frac{\alpha_l^3}{h} \right] + \frac{f(s_j)}{\sqrt{h}}, \quad j=1, 2, \ldots, n,
\end{equation}
where $\alpha_l= \pi_n x_n^G(s_l)$ and $s_l = (l-\frac{1}{2})h$ for $l = 1, 2, \ldots, n$. The above system is obtained by replacing all the integrals by numerical integration formula.

We have used Picard's iteration to solve the above system of non-linear equations.

Let $t_i=(i-1)/20, i=1,2,\ldots,21$ be the partition points with step size $h=\frac{1}{20}.$
It is easy to see that 
\begin{align*}
	E_1^n(t_i) =|\varphi(t_i) - x_n^S(t_i)| = O\left( h^2 \right).
\end{align*}
We define 
\begin{align*}
	x_n^{EX}(t_i) = \frac{4 x_{2n}^S(t_i) -x_n^S(t_i)}{3}.
\end{align*}
Then $$ E_2^n(t_i) = \left|\varphi(t_i) - x_n^{EX}(t_i) \right| = O\left( h^4 \right).$$
The orders of convergence are calculated using the formula :
\begin{align*}
\begin{aligned}
	  \alpha_1= \frac{log(E_1^n(t_i)/E_1^{2n}(t_i))}{log(2)}, \\ 
	 \beta=\frac{log(E_2^n(t_i)/E_2^{2n}(t_i))}{log(2)},  
\end{aligned} ~~~ ~~~~ n=40	
\end{align*}

\begin{align*}
	\alpha_2= \frac{log(E_1^n(t_i)/E_1^{2n}(t_i))}{log(2)}, \quad n=20.
\end{align*}

We expect $\alpha_1=\alpha_2=2$ and $\beta=4.$

\begin{center}
	Table 4.1
	
	\begin{tabular} {|c|c|c|c|c|c|}\hline
		$t_i$ & $E_1^n(t_i): n=20$ & ~~~ $E_1^n(t_i): n=40$ & $E_1^n(t_i): n=80$ & ~~~ $\alpha_{1}$ & ~~~ $\alpha_{2}$
		\\
		\hline
		$0.05$&  $  8.6 \times 10^{-3}	 $ &  $  2.15 \times 10^{-3} $ &  $  5.37 \times 10^{-4} $ & $2.00$ & $2.00$               \\
		$0.1$&  $  7.56 \times 10^{-3}	 $ &  $  1.89 \times 10^{-3} $ &  $  4.72 \times 10^{-4} $ & $2.00$ & $2.00$               \\
		$0.15$&  $  6.79 \times 10^{-3}	 $ &  $  1.7 \times 10^{-3} $ &  $  4.24 \times 10^{-4} $ & $2.00$ & $2.00$               \\
		$0.2$&  $  6.22 \times 10^{-3}	 $ &  $  1.55 \times 10^{-3} $ &  $  3.89 \times 10^{-4} $ & $2.00$ & $2.00$               \\
		$0.25$&  $  5.78 \times 10^{-3}	 $ &  $  1.44 \times 10^{-3} $ &  $  3.61 \times 10^{-4} $ & $2.00$ & $2.00$               \\
		$0.3$&  $  5.45 \times 10^{-3}	 $ &  $  1.36 \times 10^{-3} $ &  $  3.4 \times 10^{-4} $ & $2.00$ & $2.00$               \\
		$0.35$&  $  5.19 \times 10^{-3}	 $ &  $  1.3 \times 10^{-3} $ &  $  3.24 \times 10^{-4} $ & $2.00$ & $2.00$               \\
		$0.4$&  $  4.98 \times 10^{-3}	 $ &  $  1.25 \times 10^{-3} $ &  $  3.11 \times 10^{-4} $ & $2.00$ & $2.00$               \\
		$0.45$&  $  4.82 \times 10^{-3}	 $ &  $  1.2 \times 10^{-3} $ &  $  3.01\times 10^{-4} $ & $2.00$ & $2.00$               \\
		$0.5$&  $  4.68 \times 10^{-3}	 $ &  $  1.17 \times 10^{-3} $ &  $  2.92 \times 10^{-4} $ & $2.00$ & $2.00$               \\
		$0.55$&  $  4.55 \times 10^{-3}	 $ &  $  1.14 \times 10^{-3} $ &  $  2.84 \times 10^{-4} $ & $2.00$ & $2.00$               \\
		$0.6$&  $  4.44 \times 10^{-3}	 $ &  $  1.11 \times 10^{-3} $ &  $  2.77 \times 10^{-4} $ & $2.00$ & $2.00$               \\
		$0.65$&  $  4.33 \times 10^{-3}	 $ &  $  1.08 \times 10^{-3} $ &  $  2.7 \times 10^{-4} $ & $2.00$ & $2.00$               \\
		$0.7$&  $  4.22 \times 10^{-3}	 $ &  $  1.05 \times 10^{-3} $ &  $  2.64 \times 10^{-4} $ & $2.00$ & $2.00$               \\
		$0.75$&  $  4.10 \times 10^{-3}	 $ &  $  1.02 \times 10^{-3} $ &  $  2.56 \times 10^{-4} $ & $2.00$ & $2.00$               \\
		$0.8$&  $  3.98 \times 10^{-3}	 $ &  $  9.94 \times 10^{-4} $ &  $  2.48 \times 10^{-4} $ & $2.00$ & $2.00$               \\
		$0.85$&  $  3.84 \times 10^{-3}	 $ &  $  9.6 \times 10^{-4} $ &  $  2.4 \times 10^{-4} $ & $2.00$ & $2.00$               \\
		$0.9$&  $  3.69 \times 10^{-3}	 $ &  $  9.22 \times 10^{-4} $ &  $  2.3 \times 10^{-4} $ & $2.00$ & $2.00$               \\
		$0.95$&  $  3.52 \times 10^{-3}	 $ &  $  8.8 \times 10^{-4} $ &  $  2.2 \times 10^{-4} $ & $2.00$ & $2.00$               \\
		\hline
		
	\end{tabular}
\end{center}

\newpage
\begin{center}
	
	Table 4.2
	
	\begin{tabular} {|c|c|c|c|}\hline
		$t_i$ & $E_2^n(t_i): n=20$ & ~~~ $E_2^n(t_i): n=40$ & ~~~ $\beta$ \\
		\hline
		$0.05$&  $  2.98 \times 10^{-6}	 $ &  $  1.87 \times 10^{-7} $  &$3.99$ \\
		$0.1$&  $  2.23 \times 10^{-6}	 $ &  $  1.41 \times 10^{-7} $  &$3.99$ \\
		$0.15$&  $  1.59 \times 10^{-6}	 $ &  $  1.01 \times 10^{-7} $  &$3.99$ \\
		$0.2$&  $  1.09 \times 10^{-6}	 $ &  $  6.94 \times 10^{-8} $  &$3.97$ \\
		$0.25$&  $  7.13 \times 10^{-7}	 $ &  $  4.58 \times 10^{-8} $  &$3.96$ \\
		$0.3$&  $  4.46 \times 10^{-7}	 $ &  $  2.91 \times 10^{-8} $  &$3.94$ \\
		$0.35$&  $  2.7 \times 10^{-7}	 $ &  $  4.58 \times 10^{-8} $  &$3.91$ \\
		$0.4$&  $  1.69 \times 10^{-7}	 $ &  $  4.58 \times 10^{-8} $  &$3.86$ \\
		$0.45$&  $  1.3 \times 10^{-7}	 $ &  $  4.58 \times 10^{-8} $  &$3.83$ \\
		$0.5$&  $  1.41 \times 10^{-7}	 $ &  $  4.58 \times 10^{-8} $  &$3.85$ \\
		$0.55$&  $  1.91 \times 10^{-7}	 $ &  $  4.58 \times 10^{-8} $  &$3.89$ \\
		$0.6$&  $  2.72 \times 10^{-7}	 $ &  $  4.58 \times 10^{-8} $  &$3.93$ \\
		$0.65$&  $  3.75 \times 10^{-7}	 $ &  $  4.58 \times 10^{-8} $  &$3.95$ \\
		$0.7$&  $  4.95 \times 10^{-7}	 $ &  $  4.58 \times 10^{-8} $  &$3.97$ \\
		$0.75$&  $  6.26 \times 10^{-7}	 $ &  $  4.58 \times 10^{-8} $  &$3.98$ \\
		$0.8$&  $  7.6 \times 10^{-7}	 $ &  $  4.58 \times 10^{-8} $  &$3.99$ \\
		$0.85$&  $  8.94 \times 10^{-7}	 $ &  $  4.58 \times 10^{-8} $  &$3.99$ \\
		$0.9$&  $  1.02 \times 10^{-6}	 $ &  $  4.58 \times 10^{-8} $  &$3.99$ \\
		$0.95$&  $  1.14 \times 10^{-6}	 $ &  $  4.58 \times 10^{-8} $  &$4$ \\
		\hline
	\end{tabular}
\end{center}

\noindent
This verifies the result \eqref{eq:3.9}. \\

\noindent
\textbf{\large Acknowledgment}

The author Akshay S. Rane would like to thank UGC faculty recharge program, India for their support.

%%%%%%%%%%%%%%%%%%%%%%%%%%%%%%%% References %%%%%%%%%%%%%%%%%%%%%%%%%%%%%%%%%%%%%%%%%

\end{document}